\documentclass[preprint,11pt]{elsarticle}

\usepackage{amsfonts, amsmath, amscd}
\usepackage[psamsfonts]{amssymb}

\usepackage{amssymb}

\usepackage{pb-diagram}

\usepackage[all,cmtip]{xy}

\usepackage[pdftex,unicode]{hyperref}

\usepackage[usenames]{color}

\newtheorem{theorem}{Theorem}[section]
\newtheorem{lemma}[theorem]{Lemma}
\newtheorem{corollary}[theorem]{Corollary}

\newtheorem{remark}[theorem]{Remark}
\newtheorem{proposition}[theorem]{Proposition}
\newtheorem{definition}[theorem]{Definition}
\newtheorem{example}[theorem]{Example}

\newtheorem{problem}[theorem]{Problem}

\newproof{proof}{Proof}

\numberwithin{equation}{section}
\numberwithin{theorem}{section}

\newcommand{\e}{\varepsilon}
\newcommand{\w}{\omega}

\newcommand{\RR}{\mathbb{R}}

\newcommand{\IR}{\mathbb{R}}

\newcommand{\UU}{\mathcal{U}}

\newcommand{\Ra}{\Rightarrow}

\newcommand{\LRa}{\Leftrightarrow}
\newcommand{\cacx}{\overline{\mathrm{acx}}}

\newcommand{\SM}{{\setminus}}

\def\om{\omega}
\def\al{\alpha}
\def\be{\beta}

\def\sm/{{s.m.}}
\def\cm/{{c.m.}}

\begin{document}

\begin{frontmatter}

\title{Completeness and reflexivity type properties of $B_1(X)$} 

\author{Saak Gabriyelyan}
\ead{saak@math.bgu.ac.il}
\address{Department of Mathematics, Ben-Gurion University of the Negev, Beer-Sheva, Israel}

\author{Alexander V. Osipov}
\ead{OAB@list.ru}
\address{Krasovskii Institute of Mathematics and Mechanics, Ural Federal University, Yekaterinburg, Russia}

\author{Evgenii Reznichenko}
\ead{erezn@inbox.ru}
\address{Department of Mathematics, Lomonosov Mosow State University, Moscow, Russia}

\begin{abstract}
For a Tychonoff space $X$,  $B_1(X)$  denotes the space of all Baire-one functions on $X$ endowed with the pointwise topology. We prove that the following assertions are equivalent: (1) $B_1(X)$ is a (semi-)Montel space, (2) $B_1(X)$ is a (semi-)reflexive space, (3) $B_1(X)$ is a (quasi-)complete space, (4) $B_1(X)=\IR^X$, (5)  $X$ is a $Q_f$-space. It is proved that $B_1(X)$ is sequentially complete iff $B_1(X)$ is locally complete  iff   $X$ is a $CZ$-space. In the case when $K$ is a compact space, we show that $B_1(K)$ is locally complete  iff $K$ is scattered. We thoroughly study the case when $X$ is a separable metrizable space. Numerous distinguished examples are given.
\end{abstract}


\end{frontmatter}


\section{Introduction}


All topological spaces are assumed to be Tychonoff. For a space $X$, $C_p(X)$ denotes the space $C(X)$ of all real-valued functions on $X$ endowed with the pointwise topology. We denote by $B_1(X)$ the space of {\em Baire-one} functions on $X$, i.e., $B_1(X)$ is the family of all functions on $X$ which are limits of sequences in $C_p(X)$.

The spaces of Baire-one functions are widely studied in general topology and functional analysis.
If $M$ is a complete separable metrizable (= Polish) space, Bourgain, Fremlin and Talagrand proved in \cite{BFT} that $B_1(M)$ is an angelic space. The compact subsets of $B_1(M)$ (called Rosenthal compact) have been studied intensively by Rosenthal \cite{Rosenthal-compact}, Godefroy \cite{Godefroy}, Todor\v{c}evi\'{c} \cite{Todorcevic} and others. Various topological properties of $B_1(X)$ over an arbitrary space $X$ are characterized in \cite{BG-Baire,Gabr-B1,Gabr-seq-Ascoli,Osipov-25,Osipov-251,Osipov-252,Pytkeev-B1}.

The space $B_1(X)$ which satisfies some of the weak barrelledness conditions, $(DF)$-type properties, the Grothendieck property, Dunford--Pettis type properties, the Josefson--Nissenzweig property and Pe{\l}czy\'{n}ski's properties $V_{(p,q)}$ and $V_{(p,q)}^\ast$ were characterized in \cite{BG-Baire-lcs,BG-JNP,Gab-Pel}. We recall only the following results because they are used below (all relevant definitions are given in the correspondent place below).
\begin{theorem} \label{t:B1-Baire-like}
Let $X$ be a Tychonoff space. Then:
\begin{enumerate}
\item[{\rm(i)}]  $B_1(X)$ is Baire-like and hence barrelled {\rm(\cite{BG-Baire-lcs})}.
\item[{\rm(ii)}] $B_1(X)$ is Baire if, and only if, $X$ has the property $(\kappa)$ {\rm(\cite{Osipov-252})}.
\item[{\rm(iii)}] If $X$ is normal, then $B_1(X)=\IR^X$ if, and only if, $X$ is a $Q$-space {\rm(\cite{BG-Baire})}.
\item[{\rm(iv)}] If $X$ has countable pseudocharacter, then $B_1(X)$ is a Choquet space if, and only if, $X$ is a $\lambda$-space {\rm(\cite{BG-Baire})}.
\end{enumerate}
\end{theorem}

In the class of all locally convex spaces, one of the most important locally convex properties are completeness and reflexivity type properties. These properties are thoroughly studied in functional analysis, see for example the classical books \cite{Jar,PB}.  The next result concerning reflexivity type properties is the strongest one known up to now.

\begin{theorem}[\protect{\cite{BG-Baire-lcs}}] \label{t:B1-Baire-complete}
For a space $X$ of countable pseudocharacter, the following assertions are equivalent:
\begin{enumerate}
\item[{\rm (a)}] $B_1(X)$ is a Montel space;
\item[{\rm (b)}] $B_1(X)$ is a semi-Montel space;
\item[{\rm (c)}] $B_1(X)$ is a reflexive space;
\item[{\rm (d)}] $B_1(X)$ is a semi-reflexive space;
\item[{\rm (e)}] $B_1(X)$ is a complete space;
\item[{\rm (f)}] $B_1(X)$ is a quasi-complete space;
\item[{\rm (g)}] $B_1(X)=\IR^X$.
\end{enumerate}
\end{theorem}
Sequential completeness and  local completeness of spaces $B_1(X)$ were not studied at all.
The main purpose of this note is to fill this gap.

Now we describe the content of the article.
In Section \ref{sec:CZ} we introduce  two new class of topological spaces. The first one is the class of $CZ$-spaces which characterize sequential completeness  and local completeness of $B_1(X)$. This class includes all functionally countable spaces (Theorem \ref{t:cz-fc}). In Theorem \ref{t:cz-compact} we show that a compact space $X$ is a $CZ$-space if, and only if, $X$ is scattered.
The second class is the class of $Q_f$-spaces which characterizes reflexivity type properties of $B_1(X)$. In Proposition \ref{p:B1-pc} we prove that a pseudocompact space $X$ is a $Q_f$-space if, and only if, $X$ is a countable metrizable compact space.

The main results are proved in Section \ref{sec:main}. In Theorem \ref{t:B1-quasi-complete} we characterize spaces $X$ for which $B_1(X)$ satisfy one of the conditions in Theorem \ref{t:B1-Baire-complete}. Sequentially complete and locally complete spaces $B_1(X)$ are characterized in Theorem \ref{t:B1-lc}. Using these theorems   we show in Corollary \ref{c:B1-lc-pseudo} that for a compact space $K$, the space $B_1(K)$ is locally complete if, and only if, $B_1(K)$ is sequentially complete  if, and only if,  $K$ is scattered. It easily follows that the space $B_1\big([0,\w_1]\big)$ is sequentially complete but not quasi-complete, see Example \ref{exa:B1-sc-not-qc}. In Corollary \ref{c:B1-lc-Choquet} we show that if $B_1(X)$ is locally complete, then it is a Choquet space (hence, Baire). But the converse is not true in general, see Example \ref{exa:B1-Choquet-not-lc}.

Being motivated by the aforementioned results and articles, we separate the case when $X$ is a separable metric space. This case is considered in the last Section \ref{sec:B1-sms}. Let $X$ be a separable metrizable space. In Theorem \ref{t:B1-sms} we show that completeness, reflexivity and Baire type properties of $B_1(X)$ are equivalent to one of the  ``thing'' properties of the space $X$ as being a countable space, a $Q$-set, a $\sigma$-set, a $\lambda$-set or a $\kappa$-sets. The last class of $\kappa$-sets is a new one. It should be emphasized that these types of spaces are one of the basic objects for studying in descriptive set theory, general topology and measure theory, see for example the books \cite{Bukovsky2011,Kechris,Kuratowski,Miller2017} and the influential articles \cite{vD,Miller}. This study is closely related to several classical small cardinals. By this reason  in Remark \ref{rem:small-sms} we discuss the relationships between small cardinals and thing sets, and, using Theorem \ref{t:B1-sms}, we show how small cardinals and set-theoretic axioms influence on completeness, reflexivity and Baire type properties of $B_1(X)$. Additional examples are given in Examples \ref{exa:gamma-neq-kappa}, \ref{exa:gamma-neq-sigma} and \ref{exa:Q-neq-sigma}. The obtained results for thing spaces and for the properties of $B_1(X)$ are summarized in Theorem \ref{t:sms} and Theorem \ref{t:sms-B1}, respectively.



\section{$CZ$-spaces and $Q_f$-spaces} \label{sec:CZ}


In this section we introduce two new class of topological spaces which play an essential role in our article. We start from the recalling some basic notions.

Let $X$ be a space. A subset $A$ of $X$ is a {\em zero-set} if there is $f\in C(X)$ such that $A=f^{-1}(0)$. A subset $B$ of $X$ is called a {\em cozero-set} (or {\em functionally open}) if $B=X\SM A$ for some zero-set $A\subseteq X$.
A countable union of zero sets is called a {\em $Zer_{\sigma}$-set} and a countable intersection of cozero set is called a {\em $Coz_{\delta}$-set}.
It is obvious that the complement to a $Coz_{\delta}$-set is a $Zer_{\sigma}$-set, and vice versa.  It is easy to see that the family of $Zer_{\sigma}$-sets is closed under taking countable unions and finite intersections.
Observe also that any zero set is a $Coz_{\delta}$-set, and each cozero-set is a $Zer_{\sigma}$-set.
A subset $A$ of $X$ is called a {\em $CZ$-set} if $A$ is a $Zer_{\sigma}$-set and $Coz_{\delta}$-set at the same time.
The next class of spaces will be important in the article.

\begin{definition} \label{def:CZ-space} {\em
A space $X$ is called a {\em $CZ$-space}  if any $Zer_{\sigma}$-set of $X$ is a $Coz_{\delta}$-set, i.e.,
any $Zer_{\sigma}$-set of $X$ is a $CZ$-set.}
\end{definition}

For an important class of spaces including all metrizable spaces, one can reformulate the property of being a $CZ$-space in a more convenient form. Recall that a space $X$ is  {\em perfectly normal} if $X$ is normal and any closed subset of $X$ is a $G_\delta$-set. A space $X$ is perfectly normal if, and only if, any closed subset of $X$ is a zero set. In such spaces it is clear that the family $Zer_\sigma(X)$ ($Coz_\delta(X)$) coincides with the class of all $F_{\sigma}$ sets (resp., $G_{\delta}$ sets). Therefore we have the following assertion.

\begin{proposition} \label{p:pn-CZ=sigma}
A perfect normal space $X$ is a $CZ$-space if, and only if, every $F_{\sigma}$ subset of $X$ is $G_{\delta}$. 
\end{proposition}

We note also the following statement. Recall that a space $X$ is called a {\em $\lambda$-space} if every countable subset of $X$ is $G_{\delta}$.
\begin{proposition} \label{p:CZ-cp=>lambda}
If $X$ is a $CZ$-space of countable pseudocharacter, then $X$ is a $\lambda$-space.
\end{proposition}

\begin{proof}
Let $A=\{a_n\}_{n\in\w}$ be a countable subset of $X$. Since $X$ has countable pseudocharacter, every point $x\in X$ is a $Coz_\delta$-set. As $X$ is a $CZ$-space, $\{x\}$ is a  $Zer_\sigma$-set. Hence $A$ is a $Zer_\sigma$-set, too. Since $X$ is a $CZ$-space, $A$ is a $Coz_\delta$-set. In particular, $A$ is $G_{\delta}$. Thus $X$ is a $\lambda$-space.\qed
\end{proof}

To characterize pseudocompact spaces which are $CZ$-spaces we need the following notions. Recall that a subspace $Y$ of a space $X$ is called
\begin{itemize}
\item {\em $C$-embedded in $X$} if every function $f\in C(Y)$ can be extended to $\bar f\in C(X)$;
\item {\em $C^\ast$-embedded in $X$} if every bounded function $f\in C(Y)$ can be extended to a bounded function $\bar f\in C(X)$;
\item {\em $G_\delta$-dense in $X$} if $Y$ has nonempty intersection with any nonempty $G_\delta$-set in $X$;
\item {\em $z$-embedded} in $X$ if, for every zero set $A$ in $Y$, there exists a zero set $B$ in $X$ such that $B \cap Y = A$.
\end{itemize}
The notion of $z$-embedded subspaces was introduced and studied in \cite{Bl-H}.


Below we give some sufficient conditions on a subspace to be $z$-embedded.
\begin{lemma} \label{l:C-embed-z-embed}
Let $Y$ be a subspace of a space $X$.
\begin{enumerate}
\item[{\rm(i)}] If $Y$ is $C^\ast$-embedded in $X$, then $Y$ is $z$-embedded in $X$.
\item[{\rm(ii)}]  If $Y$ is dense, then $Y$ is $C$-embedded in $X$ if, and only if, $Y$ is $z$-embedded and $G_\delta$-dense in $X$.
\end{enumerate}
\end{lemma}

\begin{proof}
(i) Let $A=f^{-1}(0)$ be a zero subset of $Y$, where $f\in C(Y)$ is bounded. Since $Y$ is $C^\ast$-embedded in $X$, there is a bounded $\bar f\in C(X)$ such that $\bar f{\restriction}_Y=f$. Set $B:={\bar f}^{-1}(0)$. Then $B$ is a zero set in $X$ such that $B \cap Y = A$.
\smallskip

(ii) Assume that $Y$ is $C$-embedded in $X$. Then $Y$, being also $C^\ast$-embedded, is $z$-embedded in $X$ by (i).
Then Theorem 6.1.4 of \cite{ArT} implies that $Y$ is $G_\delta$-dense in $X$.



Conversely, assume that $Y$ is $z$-embedded and $G_\delta$-dense in $X$. Sets $A,B\subset X$ are {\em completely separated} in $X$ if there exists a continuous function $f$ on $X$ such that $f(A)=\{0\}$ and $f(B)=\{1\}$. The Blair--Hager theorem (\cite[Corollary 3.6.B]{Bl-H}, see also \cite[Theorem 9.9.36]{ArT}) implies that $Y$ is $C$-embedded in $X$ if and only if $Y$ is completely separated in $X$ from every nonempty zero-set disjoint from it. Since $Y$ is $G_\delta$-dense in $X$, there are no nonempty zero sets disjoint from $Y$. Thus, by the Blair--Hager theorem, $Y$ is $C$-embedded in $X$.\qed
\end{proof}

\begin{proposition}\label{p:zem}
Let $Y$ be a $z$-embedded  subspace  in a space $X$. If $M\subseteq Y$  is a $Coz_{\delta}$-set {\rm(}$Zer_{\sigma}$-set{\rm)}, then there exists a $Coz_{\delta}$-set {\rm(}resp., a $Zer_{\sigma}$-set{\rm)} $L$ in $X$ such that $L \cap Y = M$. Consequently, if $X$ is a $CZ$-space, then also $Y$ is a $CZ$-space.
\end{proposition}

\begin{proof}
By the symmetry between $Coz_{\delta}$-sets and $Zer_{\sigma}$-sets, it suffices to prove only the case of $Zer_{\sigma}$-sets. Let $M=\bigcup_{n\in\w} A_n$ be a $Zer_{\sigma}$-set in $Y$, where all $A_n$ are zero sets in $Y$. Since $Y$ is $z$-embedded, for every $n\in\w$, there is a continuous function $f_n:X\to[0,1]$ such that the set $B_n:=f_n^{-1}(0)$ satisfies the equality $A_n=Y\cap B_n$. Set $L:=\bigcup_{n\in\w} B_n$. Then $L$ is a $Zer_{\sigma}$-set in $X$ such that $L\cap Y=\bigcup_{n\in\w} (B_n\cap Y)=\bigcup_{n\in\w} A_n=M$, as desired.\qed
\end{proof}

If $Y$ is dense and $C$-embedded, also the converse in Proposition \ref{p:zem} is true.

\begin{proposition} \label{l:zcembed}
Let  $Y$ be a dense $C$-embedded subspace of a space $X$. Then the space $Y$ is a $CZ$-space if, and only if, $X$ is a $CZ$-space.
\end{proposition}

\begin{proof}
Taking into account (ii) of Lemma \ref{l:C-embed-z-embed}, the sufficiency follows from Proposition \ref{p:zem}.
\smallskip

To prove the necessity, assume that $Y$ is a $CZ$-space. Fix an arbitrary $Zer_{\sigma}$-set  $M$ in $X$. Then $M':=M\cap Y$ is a $Zer_{\sigma}$-set in $Y$. Set $C':=Y\setminus M'$. Since $Y$ is a $CZ$-space, the set $C'$ is a $Zer_{\sigma}$-set in $Y$. Hence, by Proposition \ref{p:zem},  there exists a $Zer_{\sigma}$-set $C$ in $X$ such that $C'=C\cap Y$.

It suffices to prove that $C=X\setminus M$. Suppose for a contradiction that the set
\[
D:=(M\cap C)\cup \big(X\setminus (M\cup C)\big)
\]
is not empty. Fix a point $x\in D$. We claim that there is a $G_\delta$-set $G$ in $X$ such that $x\in G\subseteq D$. Indeed, assume that $x\in M\cap C$. Since $M$ and $C$ and hence also $M\cap C$ are $Zer_{\sigma}$-sets in $X$, there is a zero set $G\subseteq M\cap C\subseteq D$ in $X$ such that $x\in G$. Observe that $G$ is also a $G_\delta$-set in $X$. Assume that $x\in X\setminus (M\cup C)$. Since $M\cap C$ is a  $Zer_{\sigma}$-set, we obtain that $x\in G:=X\setminus (M\cup C)$ is a $G_\delta$-set in $X$. This proves the claim.

Since $Y$ is dense and $C$-embedded, (ii) of Lemma \ref{l:C-embed-z-embed} implies that $Y$ is $G_\delta$-dense in $X$. Therefore,  we have $G\cap Y\neq\emptyset$. Since $G\subseteq D$ we obtain
\[
Y\cap D=(M'\cap C')\cup \big(Y\setminus (M'\cup C')\big)\neq\emptyset,
\]
that is, $C'\neq Y\setminus M'$. This contradiction finishes the proof.\qed
\end{proof}

Since any space $X$ is dense and $C$-embedded in its realcompactification $\upsilon X$, Proposition \ref{l:zcembed} implies the following assertion.

\begin{proposition}\label{p:cz-nu}
A space $X$ is a $CZ$-space if, and only if,  its realcompactification $\upsilon X$ is a $CZ$-space.
\end{proposition}

In the next theorem we show that the widely studied class of functionally countable spaces belongs to the class of $CZ$-spaces. Recall that a space $X$ is called {\em functionally countable} if  any second countable continuous image of $X$ is countable. The class of functionally countable spaces is sufficiently large, it contains all ordinals, $\sigma$-products of Cantor cubes, and Lindel\"{o}f scattered spaces. It is well-known that a compact space is functionally countable if, and only if, it is scattered.  A space $X$ is {\em scattered} if every nonempty subspace of $X$ has an isolated point.

\begin{theorem}\label{t:cz-fc}
A functionally countable space $X$ is a $CZ$-space.
\end{theorem}

\begin{proof}
To show that $X$ is a $CZ$-space, we check that every $Zer_{\sigma}$ subset of $X$ is $Coz_{\delta}$. Let $A=\bigcup_{n\in\w} A_n$ be a  $Zer_{\sigma}$ set in $X$, where $\{A_n\}_{n\in\w}$ is an increasing sequence of zero-sets in $X$. For every $n\in\w$, let $f_n: X\to \IR$ be a continuous function such that $A_n=f^{-1}_n(0)$. Consider the diagonal mapping $F=\triangle f_n: X\to \IR^{\w}$. As $\{A_n\}_{n\in\w}$ is increasing we obtain
\[
F(A)=\big\{(z_0,\dots,z_n,0,0,\dots)\in\IR^\w: n\in\w \mbox{ and } z_0\cdots z_n\not=0\big\} \cup\{(0,0,\dots)\}.
\]
It follows that $A=F^{-1}(F(A))$. Since $X$ is functionally countable, $F(X)$ is countable. Therefore $F(X)\setminus F(A)$ is countable, too. Hence, $F(A)$ is $G_{\delta}$ in $F(X)$. Therefore, by Proposition \ref{p:pn-CZ=sigma}, $F(A)=\bigcap_{n\in\w} U_n$, where all $U_n$ are cozero sets in the countable metric space $F(X)$. It follows that $A=F^{-1}(F(A))=\bigcap_{n\in\w} F^{-1}(U_n)$ is $Coz_{\delta}$.\qed
\end{proof}

Now we characterize compact spaces which are $CZ$-spaces.
\begin{theorem}\label{t:cz-compact}
A compact space $X$ is a $CZ$-space if, and only if, $X$ is scattered.
\end{theorem}

\begin{proof}
Assume that $X$ is a $CZ$-space, and suppose for a contradiction that $X$ is not scattered. Then there is a continuous function $f$ from $X$ onto the unit interval $[0,1]$. Let $Q:=\mathbb{Q}\cap[0,1]$ be the rational numbers in $[0,1]$. Then $f^{-1}(Q)$, being a $Zer_\sigma$-set, is a $CZ$-set in $X$. Hence, $X\setminus f^{-1}(Q)$ is $F_{\sigma}$. Since $X$ is compact, $X\setminus f^{-1}(Q)$ is $\sigma$-compact. Then $f(X\setminus f^{-1}(Q))=[0,1]\setminus Q$ is $F_{\sigma}$,  a contradiction.
\smallskip

Since $X$ is a scattered compact space, $X$ is functionally countable (see \cite{Arh}, \S 3). Thus, by Theorem \ref{t:cz-fc}, $X$ is a $CZ$-space.\qed
\end{proof}


The next characterization of pseudocompact $CZ$-spaces will be used in Section \ref{sec:main}. Recall that $\beta X$ denotes the Stone--\v{C}ech compactification of a space $X$.

\begin{theorem}\label{t:cz-pseudocompact} For a pseudocompact space $X$ the following assertions are equivalent:
\begin{enumerate}
\item[{\rm(i)}] $X$ is  a $CZ$-space;
\item[{\rm(ii)}] $\beta X$ is scattered;
\item[{\rm(iii)}] $\beta X$ is functionally countable;
\item[{\rm(iv)}] $X$ is functionally countable.
\end{enumerate}
\end{theorem}

\begin{proof}
Taking into account that the realcompactification $\upsilon X$ of a pseudocompact space $X$ is exactly $\beta X$ (see \cite[3.11.C]{Eng}), Proposition \ref{p:cz-nu} and Theorem \ref{t:cz-compact} imply that  $X$ is a $CZ$-space if, and only if, $\beta X$ is scattered. This proves the equivalence (i)$\LRa$(ii). The equivalences (ii)$\LRa$(iii)$\LRa$(iv) are proved in Proposition 4.5 of \cite{Choban-97}.\qed
\end{proof}

The condition that $\beta X$ is scattered in Theorem \ref{t:cz-pseudocompact} cannot be replaced by the condition that $X$ is itself scattered as the following example shows.
\begin{example}\label{Examp2}
There is a scattered pseudocompact space $X$ which is not a $CZ$-space.
\end{example}

\begin{proof}
Let $X$ be a Mr\'{o}wka--Isbell space such that $\beta X\setminus X$ is homeomorphic to the interval $[0,1]$, for a such space see \cite[Theorem 8.6.2]{HHH2018pc}. In particular, $\beta X$ is not scattered. Then the space $X$ is pseudocompact, locally countable, locally metrizable, locally compact, and scattered. However, by Theorem \ref{t:cz-pseudocompact}, the space $X$ is not a $CZ$-space.\qed
\end{proof}


 The next class of spaces will play an essential role to characterize reflexive type properties of $B_1(X)$, see Theorem \ref{t:B1-quasi-complete}.

\begin{definition} \label{def:Qf-space} {\em
A space $X$ is called a {\em $Q_f$-space}  if each subset of X is a $CZ$-set.}
\end{definition}

The following statement follows from the definitions  of $CZ$-spaces and $Q_f$-spaces.

\begin{proposition}\label{p:Qf-CZ}
Any $Q_f$-space is a $CZ$-space.
\end{proposition}

We shall use repeatedly the following well-known fact (see Exercise 3.A.1 in \protect\cite{lmz1}):  $f\in B_1(X)$ if, and only if, $f^{-1}(U)\in Zer_{\sigma}$ for every open $U\subseteq \IR$. Consequently,  the characteristic function $\mathbf{1}_A$ of a set $A\subseteq X$ is Baire one if, and only if, $A$ is a $CZ$-set.

Recall that a space $X$ is called a {\em $Q$-space} if every subset of $X$  is of type $F_\sigma$ in $X$.

\begin{proposition}\label{p:Qf-def}
For a space $X$, the following assertions are equivalent:
\begin{enumerate}
\item[{\rm (i)}] $X$ is a $Q_f$-space;
\item[{\rm (ii)}] $X$ is a $Q$-space and each closed subset of $X$ is a $CZ$-set;
\item[{\rm (iii)}]
$B_1(X)=\RR^X$.
\end{enumerate}
\end{proposition}
\begin{proof}
(i)$\Ra$(ii) follows from the fact that every $CZ$-set is an $F_\sigma$-set.
\smallskip

(ii)$\Ra$(i) Let $M$ be an arbitrary subset of $X$. Since $X$ is a $Q$-space, we have $M=\bigcup_{n\in\om} F_n$, where each $F_n$ is a closed set. Then each $F_n$ is a $CZ$-set
and $M$ is a $Zer_\sigma$-set. Analogously $X\SM M$ is a  $Zer_\sigma$-set. Thus, $M$ is a $CZ$-set.

(i)$\Ra$(iii)
Let $f:X\to\IR$ be an arbitrary function. For every open set $U\subseteq \IR$, by (i), the set $f^{-1}(U)$ is a $CZ$-set and hence it is a $Zer_{\sigma}$-set. Thus $f\in B_1(X)$.
\smallskip

(iii)$\Ra$(i)
Let $M$ be an arbitrary subset of $X$. Since $\mathbf{1}_M\in B_1(X)$, $M$ is  a $CZ$-set.
\end{proof}

Recall that a space $X$ is {\em cleavable} (some authors use the term {\em splittable}) if for every subset $A$ of $X$, there is a continuous surjective mapping $f:X\to\IR^\w$ such that the sets $f(A)$ and $f(X\SM A)$ are disjoint. We refer the reader to the survey of Arhangel'skii \cite{Arh1993}.
It is known (see \cite[Theorem~3.5]{Arh1993}) that $X$ is cleavable if, and only if, for any $f \in \IR^X$ there exists a countable $A \subseteq C_p(X)$ such that $f \in \overline{A}$.
This result and Proposition \ref{p:Qf-def} imply the following assertion.

\begin{proposition}\label{p:cleavable}
Any $Q_f$-space is cleavable.
\end{proposition}

It follows from Theorem \ref{t:B1-Baire-like} that if $X$ is a normal $Q$-space, then $B_1(X)=\IR^X$. This result and Proposition \ref{p:Qf-def} imply the next proposition.

\begin{proposition}\label{p:nQ-Qf}
Any normal $Q$-space is a $Q_f$-space.
\end{proposition}

Following  Arhangel'skii (see \cite{Arh1993}), a space $X$ is said to be {\em weakly normal} if for every two disjoint closed subsets $A$ and $B$ of $X$, there exists a continuous mapping $f:X\to\IR^\w$ such that the sets $f(A)$ and $f(X\SM A)$ are disjoint. By \cite[Theorem~3.6]{Arh1993}, every cleavable space is weakly normal.
This result and Proposition \ref{p:Qf-def} imply the following statement.

\begin{proposition}\label{p:nQ-wn}
Any $Q_f$-space is weakly normal.
\end{proposition}

Propositions \ref{p:nQ-Qf} and \ref{p:nQ-wn} motivate the following question.

\begin{problem}
Is there a non-normal $Q_f$-space?
\end{problem}

Using some results on cleavable spaces we select the next assertion.

\begin{proposition}\label{p:B1-cleavable}
Let $X$ be a $Q_f$-space. Then:
\begin{enumerate}
\item[{\rm(i)}] if $X$ is Lindel\"{o}f, then $X$ is submetrizable;
\item[{\rm(ii)}] if the cardinality of $X$ is an Ulam nonmeasurable cardinal, then $X$ is realcompact.
\end{enumerate}
\end{proposition}

\begin{proof}
By Proposition \ref{p:cleavable}, the space $X$ is cleavable. Thus the  clauses (i) and (ii) are valid by Theorems 3.1  and 3.4 of \cite{Arh1993}, respectively. \qed
\end{proof}

\begin{proposition}\label{p:B1-pc}
A a pseudocompact space $X$ is a $Q_f$-space if, and only if, it is a countable metrizable compact space.
\end{proposition}

\begin{proof}
Assume that $X$ is a $Q_f$-space. By Proposition \ref{p:cleavable}, the space $X$ is cleavable. Therefore, by Theorems 3.3 of \cite{Arh1993}, the space $X$ is a metrizable compact space.
 Since $X$ is also a $CZ$-space (Proposition \ref{p:Qf-CZ}), Theorem \ref{t:cz-pseudocompact} implies that $X$ is scattered. Being a scattered metrizable compact space $X$ is countable. The converse assertionis follows from Proposition \ref{p:pn-CZ=sigma}.\qed
\end{proof}

Below we formulate two old problems about cleavable spaces (see \cite{Arh1993,ArhSha1988}) adopted to $Q_f$-spaces.

\begin{problem}
Is it true that if $X$ is a $Q_f$-space, then $X$ is Dieudonne complete?
\end{problem}

\begin{problem}
Is it true that if $X$  is a $Q_f$-space, then the diagonal of the square $X\times X$ is a $G_\delta$-set?
\end{problem}

Below we give an example of $Q$-spaces. Recall that a space X is called {\em strictly $\sigma$-discrete} if X is the union of a countable family of closed discrete subspaces.

\begin{proposition}\label{p:ssd-Q}
Any strictly $\sigma$-discrete space is a $Q$-space.
\end{proposition}
\begin{proof}
Let $X=\bigcup_{n\in\om} X_n$, where all $X_n$ are closed discrete subspaces of $X$. Let $M\subseteq X$. Then $M=\bigcup_{n\in\w} M_n$, where each $M_n :=M\cap X_n$ is closed.\qed
\end{proof}

Propositions \ref{p:nQ-Qf} and \ref{p:ssd-Q} imply the following assertion.
\begin{proposition}\label{p:ssd-Qf}
Each normal strictly $\sigma$-discrete space is a $Q_f$-space.
\end{proposition}

\begin{example}\label{exa:pcs-Qf}{\em
Let $X$ be a pseudocompact Mr\'owka-Isbell space \cite{HH2018}. Since any Mr\'owka-Isbell space is strictly $\sigma$-discrete, by Proposition \ref{p:ssd-Q}, $X$ is a $Q$-space.
Proposition \ref{p:B1-pc} implies that $X$ is not a $Q_f$-space.
Therefore, $X$ is a separable pseudocompact $Q$-space, which is neither a normal space nor a $Q_f$-space.
}\end{example}

Recall that the {\em Suslin number} or the {\em cellularity} $c(X)$ of a space $X$ is the cardinal number
\[
c(X):=\sup\big\{ |\mathcal{V}| : \mathcal{V} \text{ is a pairwise disjoint collection of non-empty open sets in }X\big\}+\w.
\]
One says that $X$ is a {\em ccc} space (ccc=countable chain condition) if $c(X) = \w$.
For any infinite cardinal $\tau$, Arhangel’skii and Shakhmatov \cite[8.6]{ArhSha1988} constructed a strictly $\sigma$-discrete normal ccc space $X$ such that $|X|\geq \tau$. This result and Proposition \ref{p:ssd-Qf} imply the following assertion.

\begin{proposition}\label{p:Gf-ccc}
For any infinite cardinal $\tau$, there exists a ccc $Q_f$-space $X$ such that $|X|\geq\tau$.
\end{proposition}


\section{Main results} \label{sec:main}


First we recall the basic notions which are used in this section. All locally convex spaces $E$ (lcs, for short) are assumed to be Hausdorff. The dual space of $E$ is denoted by $E'$. The value of $\chi\in E'$ on $x\in E$ is denoted by $\langle \chi,x\rangle$. The space $E'$ endowed with the strong topology $\beta(E',E)$ is denoted by $E'_\beta$. Set $E'' :=(E'_\beta)'_\beta$. Denote by $\psi_E: E\to E''$ the canonical evaluation inclusion defined by $\langle \psi_E(x),\chi\rangle:=\langle \chi,x\rangle$ for all $x\in E$ and $\chi\in E'$.
Recall that an lcs $E$ is called
\begin{itemize}\itemsep=1pt\parskip=1pt
\item {\em semi-reflexive} if $\psi_E$ is surjective;
\item {\em reflexive} if $\psi_E$ is a topological isomorphism;
\item {\em semi-Montel} if every bounded subset of $E$ is relatively compact;
\item {\em Montel} if $E$ is semi-Montel and reflexive;
\item {\em complete} if each Cauchy net in $E$  converges;
\item {\em quasi-complete} if each closed bounded subset of $E$ is complete;
\item {\em sequentially complete} if each Cauchy sequence in $E$ converges;
\item {\em locally complete} if the closed absolutely convex hull of a null sequence in $E$ is compact.
\end{itemize}

Let $X$ be a space. If $\alpha>1$ is a countable ordinal, the  {\em space $B_\alpha(X)$ of Baire-$\alpha$ functions} is the family of all functions  $f:X\to \IR$ such that  there exists a sequence $\{ f_n\}_{n\in\w} \subseteq \bigcup_{i<\alpha} B_i(X)$ which pointwise converges to $f$. The spaces $B_\alpha(X)$ are endowed with the pointwise topology induced from the direct product $\IR^X$. The family $B(X):=\bigcup\{B_{\alpha}(X): \alpha\in \w_1\}$ is called the {\em space of Baire functions on $X$}. We shall use the following assertion.
\begin{proposition} \label{p:B(X)-seq-compl}
For every space $X$, the space $B(X)$ is sequentially complete.
\end{proposition}

\begin{proof}
Let $\{f_n\}_{n\in\w}$ be a Cauchy sequence in $B(X)$. Take a countable ordinal $\alpha$ such that $\{f_n\}_{n\in\w}$ is contained in $B_\alpha(X)$. Then, by definition, the limit function $f(x)=\lim_n f_n(x)$ belongs to $B_{\alpha+1}(X)\subseteq B(X)$. Thus $B(X)$ is a sequentially complete space.\qed
\end{proof}

We need the next lemma. Recall that a space $X$ is said to have {\em countable pseudocharacter} if every singleton $\{x\}\subseteq X$ is the intersection of a countable family of open subsets of $X$.
\begin{lemma} \label{l:B1-qc}
If $B_1(X)$ is a quasi-complete space, then $X$ has countable pseudocharacter.
\end{lemma}
\begin{proof}
Fix an arbitrary point $z\in X$, and let $\UU=\{U: z\in U\}$ be an open neighborhood base at the point $z$. For every $U\in\UU$, fix a continuous function $f_U:X\to [0,1]$ such that $f_U(z)=1$ and $f_U(X\SM U) \subseteq \{0\}$. Consider the set
\[
B:=\{ f\in B_1(X): |f(x)|\leq 1 \mbox{ for every } x\in X\}.
\]
It is clear that $B$ is a closed, absolutely convex and bounded subset of $B_1(X)$. Since $\{f_U\}_{U\in\UU}$ is a Cauchy net in $B$, the quasi-completeness of $B_1(X)$ implies that $B$ is compact and, hence, $\{f_U\}_{U\in\UU}$ has a complete accumulation point $f\in B\subseteq B_1(X)$. As $\UU$ is a neighborhood base at $z$, we must have $f=\mathbf{1}_{\{z\}}\in B_1(X)$. It follows that $\{z\}$ is a $CZ$-set. Hence there is a sequence $\{V_n\}_{n\in\w}$ of open neighborhoods of $z$ such that $\{z\}=\bigcap_{n\in\w} V_n$. Thus $X$ has countable pseudocharacter at $z$, as desired.\qed
\end{proof}

Now we are ready to prove the first main result of the article, which
 together with Lemma \ref{l:B1-qc} shows that the condition on $X$ to have countable pseudocharacter in Theorem \ref{t:B1-Baire-complete} can be omitted.
\begin{theorem} \label{t:B1-quasi-complete}
For a space $X$, the following assertions are equivalent:
\begin{enumerate}
\item[{\rm(i)}] $B_1(X)$ is a Montel space;
\item[{\rm(ii)}] $B_1(X)$ is a semi-Montel space;
\item[{\rm(iii)}] $B_1(X)$ is a reflexive space;
\item[{\rm(iv)}] $B_1(X)$ is a semi-reflexive space;
\item[{\rm(v)}] $B_1(X)$ is a complete space;
\item[{\rm(vi)}] $B_1(X)$ is a quasi-complete space;
\item[{\rm(vii)}] $B_1(X)=\IR^X$;
\item[{\rm(viii)}] $X$ is a $Q_f$-space.
\end{enumerate}
\end{theorem}

\begin{proof}
Since $B_1(X)$ is barrelled (Theorem \ref{t:B1-Baire-like}) and $B_1(X)$ carries its weak topology, the equivalences (i)$\LRa$(ii)$\LRa$(iii)$\LRa$(iv)$\LRa$(vi) follow from \cite[Proposition~11.4.2]{Jar}. The  implication (v)$\Ra$(vi) is trivial.
\smallskip

(vi)$\Ra$(v) and (vi)$\Ra$(vii): By Lemma \ref{l:B1-qc}, the space $X$ has countable pseudocharacter. Therefore, both  implications follow from Theorem \ref{t:B1-Baire-complete}.
\smallskip

(vii)$\Ra$(viii) For every subset $A$ of $X$, the equality $B_1(X)=\IR^X$ implies that $\mathbf{1}_A$ is a Baire-one function. Thus $A$ is a $CZ$-set.
\smallskip

(vii)$\LRa$(viii) follows from Proposition \ref{p:Qf-def}.
\qed
\end{proof}

Now we prove the second main result of the article.

\begin{theorem}\label{t:B1-lc}
For a space $X$, the following assertions are equivalent:
\begin{enumerate}
\item[{\rm(i)}] $B_1(X)$ is sequentially complete;
\item[{\rm(ii)}] $B_1(X)$ is locally complete;
\item[{\rm(iii)}] $B_1(X)=B_2(X)$;
\item[{\rm(iv)}]   $X$ is a $CZ$-space.
\end{enumerate}
\end{theorem}

\begin{proof}
(i)$\Rightarrow$(ii) is well-known, see  Corollary 5.1.8 in \cite{PB}.
\smallskip

(ii)$\Rightarrow$(iv) Suppose for a contradiction that  $X$ is not a $CZ$-space. Taking into account the symmetry between  $Zer_{\sigma}$-sets and $Coz_{\delta}$-sets, we can assume that there is a $Zer_{\sigma}$-set  $A$  in $X$ which is not a $Coz_{\delta}$-set. Let $A=\bigcup\{A_i: i\in \w\}$, where  for each $i\in \w$, $A_i$ is a zero-set in $X$ and $A_i\subseteq A_{i+1}$. Note that for every $i\in\w$, the set $B_{i+1}:=A_{i+1}\setminus A_i=A_{i+1} \cap(X\SM A_i)$ is a $CZ$-set. Set $B_0:=A_0$, so that $B_0$ is also a $CZ$-set and $A=\bigcup\{B_i: i\in \w\}$. Therefore we can consider the sequence $\{g_i\}_{i\in \omega}$ of Baire-one functions, where
\[
 g_i:= 2^{i+1} \cdot \mathbf{1}_{B_i} \; \mbox{ for every $i\in\w$}.
\]
Since $B_i\cap B_j=\emptyset$ for all distinct $i,j\in\w$,  $\{g_i\}_{i\in\w}$ is a null sequence in $B_1(X)$. For every $n\in \w$, set
\[
F_n(x):=\sum\limits_{i\leq n} \tfrac{1}{2^{i+1}}\cdot g_i(x)=\sum\limits_{i\leq n} \mathbf{1}_{B_i}(x)\in \cacx(\{g_i\}_{i\in\omega}).
\]
Since $B_1(X)$ is locally complete, the set $\cacx(\{g_i\}_{i\in\omega})$ is compact, and hence the sequence $\{F_n\}_{n\in\w}$ has a cluster point $F\in B_1(X)$. By construction it is clear that $F(x)=0$ for every $x\in X\SM A$. If $x \in B_i$ for some $i\in\w$, then
\[
F_n(x)=\tfrac{1}{2^{i+1}}\cdot g_i(x) = 1 \quad \mbox{ for every $n>i$}.
\]
Therefore $F(x)=1$ for every $x\in A=\bigcup_{i\in\w} B_i$. Hence $F(x)=\mathbf{1}_A$. As $F(x)\in B_1(X)$, it follows that $A$ is a $Coz_{\delta}$-set (even a $CZ$-set). This is a desired contradiction.
\smallskip

(iv)$\Rightarrow$(iii) Assume that $X$ is a $CZ$-space. Then the equality $Coz_{\delta}(X)= Zer_{\sigma}(X)$ implies that  any Baire subset of $X$ is a $CZ$-set. In particular,  $B_1(X)=B_2(X)$.
\smallskip

(iii)$\Rightarrow$(i) Since $B_1(X)=B_2(X)$, we obtain that  $B_1(X)=\bigcup\{B_{\alpha}(X): \alpha\in \omega_1\}=B(X)$, i.e., any Baire function is a Baire-one function. It remains to note that, by Proposition \ref{p:B(X)-seq-compl}, the space $B(X)$ is sequentially complete.\qed
\end{proof}

Theorems \ref{t:cz-compact} and \ref{t:cz-pseudocompact} and Theorem \ref{t:B1-lc} imply the following assertion.
\begin{corollary} \label{c:B1-lc-pseudo}
If $X$ is a pseudocompact space, then $B_1(X)$ is locally complete if, and only if, $\beta X$ is scattered. In particular, for a compact space $K$, the space $B_1(K)$ is locally complete if, and only if, $B_1(K)$ is sequentially complete  if, and only if,  $K$ is scattered.
\end{corollary}

The next example shows that, in the class of Baire-one functions, sequential completeness is strictly weaker than quasi-completeness.
\begin{example} \label{exa:B1-sc-not-qc}
The space $B_1\big([0,\w_1]\big)$ is sequentially complete but not quasi-complete.
\end{example}

\begin{proof}
Since the compact space $[0,\w_1]$ is scattered, Corollary \ref{c:B1-lc-pseudo} implies that $B_1\big([0,\w_1]\big)$ is sequentially complete. On the other hand, since $\{\w_1\}$ is not $G_\delta$, Theorem \ref{t:B1-quasi-complete} implies that $B_1\big([0,\w_1]\big)$ is not quasi-complete.\qed
\end{proof}

It follows from Corollary 5.4 
of \cite{GOR-kFU} that if $B_1(X)$ is locally complete, then it is a Baire space. It turns out that a stronger result holds true. Recall that the property of being a Choquet space is stronger than the Baire property. We refer the reader to \cite{BG-Baire} for numerous results and historical remarks about the Choquet property. Spaces $X$  for which $B_1(X)$ is a Choquet space are completely characterized in Theorem 4.5 of \cite{Osipov-251}. The next corollary gives a sufficient condition and connects local completeness with the  Choquet property in the class of Baire-one functions.

\begin{corollary} \label{c:B1-lc-Choquet}
If $B_1(X)$ is locally complete, then it is a Choquet space and, hence, Baire.
\end{corollary}

\begin{proof}
By Theorem \ref{t:B1-lc}, we have $B_1(X)=B_2(X)$. Therefore, by Corollary 4.2 of \cite{BG-Baire}, $B_1(X)$ is a Choquet space.\qed
\end{proof}

It is natural to ask whether the converse in Corollary \ref{c:B1-lc-Choquet} holds true, namely: {\em Is it true that if $B_1(X)$ is a Choquet space, then $B_1(X)$ is locally complete}? We answer this question in the negative.

\begin{example}\label{exa:B1-Choquet-not-lc}
There is a separable pseudocompact space $X$ such that $B_1(X)$ is a Choquet space which is not locally complete.
\end{example}

\begin{proof}
Let $X$ be a Mr\'{o}wka--Isbell space from Example \ref{Examp2}. By Theorem 2.9  in \cite{KKL} and Theorem 3.1 in \cite{Osipov-25}, $B_1(X)$ is a Choquet space. Since $X$ is not a $CZ$-space, Theorem \ref{t:B1-lc} implies that  $B_1(X)$ is not locally complete.\qed
\end{proof}

It is worth mentioning that there are Baire  spaces $B_1(X)$ which are not Choquet.
\begin{example}[{\cite[Corollary 4.6]{Osipov-25}}]\label{exa:B1-Baire-not-Choquet}
There is a separable pseudocompact space $X$ such that $B_1(X)$ is a Baire  space which is not Choquet.
\end{example}

We finish this section with the following natural open problem.

\begin{problem}
Let $\alpha>1$ ba a countable ordinal. Characterize spaces $X$ for which $B_\alpha(X)$ has one of the properties from Theorems \ref{t:B1-quasi-complete} and \ref{t:B1-lc}. Analogously for the space $B(X)$.
\end{problem}


\section{Baire-one functions on separable metrizable spaces} 
\label{sec:B1-sms}


In this section, we consider the important case of separable  metrizable spaces. It turns out that, in the realm of separable metrizable spaces, the results obtained in the previous section lead us to some classical classes of separable metric spaces  which were considered, for example, in \cite{Miller,Bukovsky2011}. This line of research  began from the study of subsets of the real line $\IR$, and  then it was expanded to subsets of Polish spaces which can be viewed as subsets of the Polish space $\IR^\w$.
In various models of $\mathrm{ZFC}$, these classes are  often countable metrizable (\cm/ for short) spaces.

A separable metrizable (\sm/ for short) space $X$ is called

\begin{itemize}
\item  a {\em $Q$-set} if each subset of $X$ is a $G_\delta$-set, i.e., $X$ is a $Q$-space;
\item  a {\em $\sigma$-set}, if each $F_\sigma$-subset of $X$ is a $G_\delta$-set, i.e., $X$ is a $\sigma$-space in the sense of Kuratowski \cite[\S~40.VI]{Kuratowski};
\item  a {\em $\lambda$-set}  if if each countable subset of $X$ is a $G_\delta$-set, i.e.,  $X$ is a $\lambda$-space;
\item  a {\em $\kappa$-set} if for every pairwise disjoint sequence $\{M_n\}_{n\in\w}$ of finite sets of $X$ there exists an infinite $A\subseteq \w$ such that the set $\bigcup_{n\in A}M_n$ is a $G_\delta$-set.
\end{itemize}

The notion of $\kappa$-set is new. Theorem 3.12 in \cite{Osipov-251} implies the following statement.
\begin{proposition}\label{p:kappa-set}
$\kappa$-sets are precisely \sm/ spaces with the property $(\kappa)$.
\end{proposition}
It is obvious that
\[
\xymatrix{
\mbox{\cm/ space}   \ar@{=>}[r] &
\mbox{$Q$-set}  \ar@{=>}[r] & \mbox{$\sigma$-set}  \ar@{=>}[r] & \mbox{$\lambda$-set} \ar@{=>}[r] & \mbox{$\kappa$-set.}
}
\]

\begin{theorem}\label{t:B1-sms}
Let $X$ be a \sm/ space. Then:
\begin{enumerate}
\item[{\rm (i)}] $B_1(X)$ is Polish  if and only if $X$ is \cm/;
\item[{\rm (ii)}] $B_1(X)$ is complete {\rm(}that is $B_1(X)=\IR^X${\rm)} if and only if $X$ is a $Q$-set;
\item[{\rm (iii)}] $B_1(X)$ is sequentially complete if and only if $X$ is a $\sigma$-set;
\item[{\rm (iv)}] $B_1(X)$ is a Choquet space if and only if $X$ is a $\lambda$-set;
\item[{\rm (v)}] $B_1(X)$ is a Baire space if and only if $X$ is a $\kappa$-set.
\end{enumerate}
\end{theorem}
\begin{proof}
(i) follows from (ii) and the fact that $B_1(X)$ is metrizable if and only if $X$ is countable.

(ii) follows from Theorem \ref{t:B1-Baire-like}(iii).

(iii) follows from Theorem \ref{t:B1-lc} and Proposition \ref{p:pn-CZ=sigma}.

(iv)  follows from Theorem \ref{t:B1-Baire-like}(iv).

(v) follows from Theorem \ref{t:B1-Baire-like}(ii) and Proposition \ref{p:kappa-set}.
\qed
\end{proof}

It follows from Theorem \ref{t:B1-lc} and Corollary \ref{c:B1-lc-Choquet} that for $B_1(X)$ we have
\[
\xymatrix{
\mbox{Polish}  \ar@{=>}[r] &
\mbox{complete}  \ar@{=>}[r] & \mbox{sequentially complete}  \ar@{=>}[r] & \mbox{Choquet}
\ar@{=>}[r] & \mbox{Baire.}
}
\]
However, the converse implications are not true. For the first implication, we take an uncountable discrete space as an example. For the remaining implications, we take Examples \ref{exa:B1-sc-not-qc}, \ref{exa:B1-Choquet-not-lc}, and \ref{exa:B1-Baire-not-Choquet}.

Theorem \ref{t:B1-sms} shows that completeness, reflexivity and Baire type properties of $B_1(X)$ for separable metrizable spaces $X$ depend on relationships between ``thing'' sets: \cm/ spaces, $Q$-sets, $\sigma$-sets, $\lambda$-sets and $\kappa$-sets.
Being  one of the basic objects for studying in descriptive set theory, general topology and measure theory and taking into account their relations to  several classical small cardinals, it is important to consider relationships between all of these significant notions. To continue our discussion let us recall the definitions of those small cardinals which will be used below.

 Recall that a family of countable sets has the {s.f.i.p.} ({\em strong finite intersection property}) if every nonempty finite subfamily has infinite intersection. The small cardinal $\mathfrak{p}$ is defined as follows
\[
\begin{aligned}
\mathfrak{p}=\min\big\{ |\mathcal{B}|: & \;\mathcal{B} \mbox{ is a family of infinite subsets of $\w$ with the s.f.i.p.,}\\
& \; \mbox{and there is no infinite $A$ such that $A\subseteq^\ast B$ for all $B\in \mathcal{B}$}\big\},
\end{aligned}
\]
where $A\subseteq^\ast B$ means that $A\SM B$ is finite.
We shall write $A\subsetneq^\ast B$ if $A\subseteq^\ast B$ and $|B\setminus A|=\om$.
Recall also that $\leq^\ast$ denotes the preorder on $\w^\w$ defined by setting $f\leq^\ast g$ if $f(n)\le g(n)$ for all but finitely many $n \in \w$.

A subset of $\w^\w$ is said to be \emph{unbounded} if it is unbounded with respect to the preorder $\leq^\ast$. The cardinal $\mathfrak{b}$ is defined by
\[
\mathfrak{b} = \min\{|B|: \text{$B$ is an unbounded subset of $\w^\w$}\}.
\]
It is well-known that
\begin{equation} \label{equ:small-card-1}
\w_1 \leq \mathfrak{p} \leq \mathfrak{b} \leq \mathfrak{c},
\end{equation}
and, in different models of $\mathrm{ZFC}$, all these inequalities can be strongly independent from each other, see \cite{vD}.
Following \cite{BaMaZd}, we define
\begin{align*}
\mathfrak{q}_0 &= \min\{ |X| : X\text{ is a separable metric space which is not a $Q$-set}\},
\\
\mathfrak{q} &= \min\{ \kappa : \text{each  separable metric  space $X$ of cardinality $|X| \geq \kappa$ is not a $Q$-set}\}.
\end{align*}
Then, by \cite[Theorem~2]{BaMaZd}, we have
\begin{equation} \label{equ:small-card-2}
\mathfrak{p} \leq \mathfrak{q}_0 \leq \mathfrak{b} \ \text{ and }\ \mathfrak{q}_0 \leq \mathfrak{q} \leq \mathfrak{c},
\end{equation}
see also \cite[Theorem~4]{BaMaZd} which  gives more details on the location of  $\mathfrak{q}_0$ and $\mathfrak{q}$ between $\w_1$ and $\mathfrak{c}$.

To find relationships between thing spaces in Theorem \ref{t:B1-sms} is a standard set-theoretical problem. In the remark below we list the most significant known results which are used in what follows.

\begin{remark}\label{rem:small-sms} {\em
All spaces in the remark are assumed to be separable and metrizable.
\smallskip

(i) Any $Q$-set has cardinality strictly less than $\mathfrak{c}$, see the discussion after Theorem 8.54 of \cite[p.~314]{Bukovsky2011}.
\smallskip

(ii) The inequalities (\ref{equ:small-card-2}) imply that if $|X|<\mathfrak{p}$, then $X$ is a $Q$-set. Accordingly, under $\mathrm{MA}$, if $|X|<\mathfrak{c}$ then $X$ is a $Q$-set \cite[Theorem 4.2]{Miller}.
\smallskip

(iii) If $|X|<\mathfrak{b}$ then $X$ is a $\sigma$-set \cite[Theorem 9.1]{vD}.
\smallskip

(iv) It is consistent that there are no uncountable $\sigma$-sets \cite[Theorem 22]{Miller1979}.
\smallskip

(v) There is a  $\lambda$-set of cardinality $\mathfrak{b}$ \cite[Theorem 9.1]{vD}.
\smallskip

(vi) In the Cohen model all uncountable $\lambda$-sets have cardinality $\om_1<\om_2=\mathfrak{c}$ \cite[Theorem 22]{Miller1991}.
\smallskip

(vii) Under $\mathrm{CH}$, there exists  a  $\lambda$-set  which is not a $\sigma$-set \cite[\S~40,~Theorem~VI.3]{Kuratowski}.
\smallskip

(viii) Under $\mathrm{CH}$, there exists an uncountable $\sigma$-set \cite[Theorem 5.7]{Miller}.
\smallskip

(ix)
It is consistent that there is a $\kappa$-set $X$ which is not a $\lambda$-set \cite[Example 4.9]{Osipov-251}.}
\end{remark}

Below we give three additional examples under different axioms.

\begin{example}\label{exa:gamma-neq-kappa}
Under $\mathfrak{p}=\mathfrak{c}$, there exists a $\kappa$-set $X$ which is not a $\lambda$-set.
\end{example}

\begin{proof}
Under $\mathfrak{p}=\mathfrak{c}$, Theorem 8.92 of \cite{Bukovsky2011} implies that there exists a \sm/ $\gamma$-space $X$ of cardinality $\mathfrak{c}$ such that $X$ is $\mathfrak{c}$-concentrated on some countable subset $C\subseteq X$, that is, $|X\setminus U|<\mathfrak{c}$ for any open $U\supseteq C$. Therefore, $C$ is not $G_\delta$-set in $X$, and hence $X$ is not a $\lambda$-set. Since $X$ is a $\gamma$-space, by  \cite[Theorem 3.2]{Sak2},  $X$ has the property $(\kappa)$. It follows from Proposition \ref{p:kappa-set} that $X$ is a $\kappa$-set.\qed
\end{proof}

\begin{example}\label{exa:gamma-neq-sigma}
Under $\mathfrak{b}=\mathfrak{c}$, there exists a $\lambda$-set $X$ which is not a $\sigma$-set.
\end{example}
\begin{proof}
Let $Y$ be a $\lambda$-set such that $|Y|=\mathfrak{b}$ (Remark \ref{rem:small-sms}(v)). Then $|Y|=\mathfrak{c}$. Let $f\colon Y\to [0,1]$ be an arbitrary  bijective map, and let $X=\{(y,f(y)): y\in Y\}\subseteq Y\times[0,1]$ be the graph of $f$. Consider the coordinate projections
\begin{align*}
\pi_1&\colon X \to Y,\ (y,x) \mapsto y,
\\
\pi_2&\colon X \to [0,1],\ (y,x) \mapsto x.
\end{align*}
Then $\pi_1$ and $\pi_2$ are continuous bijective maps. Since $\pi_2$ is a continuous onto map, the Reclaw theorem \cite[Theorem 3.5]{Miller2017} implies that $X$ is not a $\sigma$-set. As $\pi_1$ is a continuous bijective map and $Y$ is a $\lambda$-set, then $X$ is $\lambda$-set by \cite[Lemma 9.3.1]{Miller}.\qed
\end{proof}

To give the next example we need the following lemma.

\begin{lemma}\label{l:p-Gdelta}
Let $X$ be a separable metrizable space. If $G\subseteq X$ is a $G_\delta$-set such that $|G|<\mathfrak{b}$, then $G$ is an $F_\sigma$-set.
\end{lemma}

\begin{proof}
Let $d$ be a metric on $X$. For every $\e>0$ and each $x\in X$, let  $B(x,\e)=\{y\in X: d(x,y)<\e\}$ be an open ball centered at $x$. Since $G$ is $G_\delta$, we have $X\setminus G=\bigcup_{n\in\w}F_n$, where all $F_n$ are closed in $X$. For every $x\in G$, define $f_x\in\w^\w$ by
\[
f_x(n):=\min \big\{m\in\w: B\big(x,\tfrac{1}{2^m}\big)\cap F_n=\emptyset\} \quad (n\in\w).
\]
Since $|G|<\mathfrak{b}$, we obtain that the family $\{f_x:x\in G\}$ is a bounded subset of $\w^\w$. Therefore there exists $f\in\w^\w$ such that $f_x\leq^\ast f$ for every $x\in G$. For every $m\in\w$, set
\[
U_m :=\bigcup\Big\{B\big(y,\tfrac{1}{2^{m+f(n)}}\big): n\in\w, y\in F_n \Big\}.
\]
Then $X\setminus G=\bigcap_{m\in\w} U_m$. Thus $G=\bigcup_{m\in\w} X\setminus U_m$ is an $F_\sigma$-set.\qed
\end{proof}

\begin{example}\label{exa:Q-neq-sigma}
Under $\mathfrak{p}=\mathfrak{c}$, there exists a $\sigma$-set $X$ such that $|X|=\mathfrak{c}$, and hence $X$ is not a $Q$-set.
\end{example}
\begin{proof}
Let $[\w]^\w$ be the set of all infinite subsets $A$ of $\w$. Identifying $A$ with the characteristic function $\mathbf{1}_A$ of $A$, we can consider $[\w]^\w$ as a subspace of the Cantor cube $2^\w$. Let $\{G_\alpha: \alpha< \mathfrak{c}\}$ be the set of all $G_\delta$-sets of $[\w]^\w$.
For every $A\in [\w]^\w$, we define
\[
S(A):=\{B\in [\w]^\w: B\subseteq^\ast A\}
=\bigcup_{n\in\w} \left\{ B\in [\w]^\w: B\subseteq \big(\{0,1,...,n\}\cup A\big)\right\}.
\]
Then the set $S(A)$ is an $F_\sigma$-set in $[\w]^\w$.

We construct a set $X=\{M_\alpha:\alpha<\mathfrak{c}\}\subseteq [\w]^\w$ such that
\begin{enumerate}
\item[(a)] $M_\alpha \subsetneq^\ast M_\beta$  for all $\beta< \alpha<\mathfrak{c}$, and
\item[(b)] $G_\alpha \cap S(M_\alpha)$ is an $F_\sigma$-set in $S(M_\alpha)$   for every $\alpha<\mathfrak{c}$.
\end{enumerate}
To this end, fix an arbitrary $M_0\in [\w]^\w$. Assume that for $\alpha<\mathfrak{c}$ we have already constructed $M_\beta$ for every $\beta<\alpha$. Since $\alpha<\mathfrak{p}$ and the family $\{M_\beta:\beta<\alpha\}$ have s.f.i.p., there exists  $M^\ast\in [\w]^\w$ such that $M^\ast \subseteq^\ast M_\beta$ for every $\beta<\alpha$.
Take $M\in [\w]^\w$ such that $M\subsetneq^\ast M^\ast$.
Lemma 5.7.1 of \cite{Miller} implies that there exists $M_\alpha\in [\w]^\w$ such that $M_\alpha \subseteq^\ast M$ and the intersection $G_\alpha \cap S(M_\alpha)$ is an $F_\sigma$-set in $S(M_\alpha)$.
The construction is completed.

The condition (a) imply that $M_\beta \neq M_\alpha$ for $\be<\al<\mathfrak{c}$.
\smallskip

Since $|X|=\mathfrak{c}$, it follows from Remark \ref{rem:small-sms}(i) that $X$ is not a $Q$-set.
\smallskip

To prove that $X$ is a $\sigma$-set, it suffices to show that $G_\alpha\cap X$ is an $F_\sigma$-set in $X$ for every $\alpha<\mathfrak{c}$. Setting $Q:=S(M_\alpha)$, we check this by showing that the sets $(G_\alpha\cap X) \cap Q$ and $(G_\alpha\cap X)\setminus Q$ are $F_\sigma$-sets in $X$.
\smallskip

Then $Q$ is an $F_\sigma$-set in $[\w]^\w$. It follows from (b) that $G_\alpha \cap Q$ is an $F_\sigma$-set in $Q$.
Consequently, $G_\alpha \cap Q$ is an $F_\sigma$-set in $[\w]^\w$, and hence $(G_\alpha\cap X) \cap Q$ is an $F_\sigma$-set in $X$.
\smallskip

To prove that also $(G_\alpha\cap X)\setminus Q$ is an $F_\sigma$-set in $X$, we set $G:=X\setminus Q$. As $Q$ is an $F_\sigma$-set in $[\w]^\w$, we obtain that $G$ is a $G_\delta$-set in $X$.
It follows from (a) that $M_\gamma \subseteq Q=S(M_\al)$ for $\gamma\geq \al$.
Therefore, $G\subseteq \{M_\beta:\beta<\alpha\}$ and, hence, $|G|<\mathfrak{p}=\mathfrak{c}$.
It follows from Remark \ref{rem:small-sms}(ii) that $G$ is a $Q$-set.
Therefore, $G\cap G_\alpha$ is a $G_\delta$-set in $G$. Since $G$ is a $G_\delta$-set in $X$, then $G\cap G_\alpha$ is a $G_\delta$-set in $X$. So, by (\ref{equ:small-card-2}),  $|G\cap G_\al|< \mathfrak{p}\leq \mathfrak{b}$. Applying Lemma \ref{l:p-Gdelta} we obtain that $G\cap G_\alpha=(G_\alpha\cap X)\setminus Q$ is an $F_\sigma$-set in $X$.\qed
\end{proof}

For the reader convenience and further references, we summarize relationships between the classes of thing spaces under some axioms.

\begin{theorem}\label{t:sms}
\begin{enumerate}
\item[{\rm (i)}] Under $\mathrm{CH}$, we have:
$\mbox{\cm/ spaces} = \mbox{$Q$-sets} \subsetneq \mbox{$\sigma$-sets}  \subsetneq \mbox{$\lambda$-sets} \subsetneq \mbox{$\kappa$-sets}$.
\item[{\rm (ii)}] Under  $\w_1<\mathfrak{p}=\mathfrak{c}$, we have:
$\mbox{\cm/ spaces} \subsetneq \mbox{$Q$-sets} \subsetneq \mbox{$\sigma$-sets}  \subsetneq \mbox{$\lambda$-sets} \subsetneq \mbox{$\kappa$-sets}$.
\item[{\rm (iii)}] It is consistent that: $\mbox{\cm/ spaces} = \mbox{$Q$-sets} = \mbox{$\sigma$-sets}  \subsetneq \mbox{$\lambda$-sets} \subseteq \mbox{$\kappa$-sets}$.
\end{enumerate}
\end{theorem}

\begin{proof}
(i) The equality and the inclusions follow from (i), (viii), and (vii) of Remark \ref{rem:small-sms} and Example \ref{exa:gamma-neq-kappa}, respectively.
\smallskip

(ii) The first inclusion follows from (ii) of Remark \ref{rem:small-sms}. The second one follows from Example \ref{exa:Q-neq-sigma}. The third inclusion follows from (\ref{equ:small-card-1}), (\ref{equ:small-card-2}) and Example \ref{exa:gamma-neq-sigma}. Example \ref{exa:gamma-neq-kappa} implies the last inclusion.
\smallskip

(iii) We consider a model of $\mathrm{ZFC}$ in which there are no uncountable $\sigma$-sets (see (iv) of Remark \ref{rem:small-sms}). In this model, the first two equalities hold.
The first inclusion follow from  (v) of Remark \ref{rem:small-sms}.
The last inclusion is trivial.\qed
\end{proof}

The following theorem summarizes local completeness, reflexivity and Baire type properties of $B_1(X)$ over separable metrizable spaces $X$ under some axioms, it immediately follows from Theorems \ref{t:B1-sms} and \ref{t:sms}.

\begin{theorem}\label{t:sms-B1}
The properties of being a Polish, complete, sequentially complete, Choquet or Baire space in the realm of Baire-one functions $B_1(X)$ over separable metrizable spaces $X$ are related to each other as follows:
\begin{enumerate}
\item[{\rm (i)}] under $\mathrm{CH}$: \hspace{2mm}
$
\mbox{Polish}  = \mbox{complete} \subsetneq  \mbox{sequentially complete}  \subsetneq  \mbox{Choquet}
 \subsetneq \mbox{Baire};
$
\item[{\rm (ii)}] under  $\om_1<\mathfrak{p}=\mathfrak{c}$: \hspace{2mm}
$
\mbox{Polish} \subsetneq \mbox{complete} \subsetneq  \mbox{sequentially complete}  \subsetneq  \mbox{Choquet}
 \subsetneq \mbox{Baire};
$
\item[{\rm (iii)}]
it is consistent that:  \hspace{2mm}
$
\mbox{Polish}  = \mbox{complete} = \mbox{sequentially complete}   \subsetneq  \mbox{Choquet}
 \subseteq \mbox{Baire}.
$
\end{enumerate}
\end{theorem}

\begin{problem}
Is there in $\mathrm{ZFC}$ a $\kappa$-set that is not a $\lambda$-set? In other words, does there exist in $\mathrm{ZFC}$ a separable  metrizable space $X$ for which $B_1(X)$ is Baire but not Choquet?
\end{problem}

\begin{problem}
Is there in $\mathrm{ZFC}$ a $\lambda$-set that is not a $\sigma$-set? In other words, does there exist in $\mathrm{ZFC}$ a separable  metrizable space $X$ for which $B_1(X)$ is Choquet but not sequentially complete?
\end{problem}

\bibliographystyle{amsplain}

\end{document}